\documentclass{amsart}%
\usepackage{amsfonts}
\usepackage{amsmath}
\usepackage{amssymb}
\usepackage{graphicx}%
\providecommand{\U}[1]{\protect\rule{.1in}{.1in}}
\newtheorem{theorem}{Theorem}
\theoremstyle{plain}

\newtheorem{corollary}{Corollary}

\newtheorem{example}{Example}

\newtheorem{lemma}{Lemma}

\newtheorem{problem}{Problem}

\newtheorem{remark}{Remark}

\numberwithin{equation}{section}
\begin{document}
\title[ ]{Nonlinear operator extensions of Korovkin's theorems}
\author{Sorin G. Gal}
\address{Department of Mathematics and Computer Science\\
University of Oradea\\
University\ Street No. 1, Oradea, 410087, Romania\\
and Academy of Romanian Scientists, Splaiul Independentei nr. 54, 050094,
Bucharest, Romania}
\email{galso@uoradea.ro, galsorin23@gmail.com}
\author{Constantin P. Niculescu}
\address{Department of Mathematics, University of Craiova\\
Craiova 200585, Romania and The Institute of Mathematics of the Romanian
Academy, Bucharest, Romania}
\email{constantin.p.niculescu@gmail.com}
\date{October 19, 2022}
\subjclass[2000]{41A35, 41A36, 41A63}
\keywords{Korovkin type theorems, monotone operator, sublinear operator, weakly
nonlinear operator, ordered Banach space, Choquet's integral}

\begin{abstract}
In this paper we extend Korovkin's theorem to the context of sequences of
weakly nonlinear and monotone operators defined on certain Banach function
spaces. Several examples illustrating the theory are included.

\end{abstract}
\maketitle

\section{Introduction}

Korovkin's theorem \cite{Ko1953}, \cite{Ko1960} provides a very simple test of
convergence to the identity for any sequence $(L_{n})_{n}$ of positive linear
operators that map $C\left(  [0,1]\right)  $ into itself: the occurrence of
this convergence for the functions $1,~x$ and $x^{2}$. In other words, the
fact that%
\[
\lim_{n\rightarrow\infty}L_{n}(f)=f\text{\quad uniformly on }[0,1]
\]
for every $f\in C\left(  [0,1]\right)  $ reduces to the status of the three
aforementioned functions. Due to its simplicity and usefulness, this result
has attracted a great deal of attention leading to numerous generalizations.
Part of them are included in the authoritative monograph of Altomare and
Campiti \cite{AC1994} and the excellent survey of Altomare \cite{Alt2010}. See
\cite{Alt2021}, \cite{Alt2021b}, \cite{Alt2022} and \cite{Popa2022} for some
very recent contributions.

At the core of Korovkin's theorem (as well as of many of its extensions) is
the nice behavior of the involved functions, reminding the property of
absolute continuity in real analysis. In what follows we refer to it as the
\emph{Korovkin absolute continuity}. For a real-valued function $f$ defined on
a subset $K$ of the Euclidean space $\mathbb{R}^{N}$ this property means the
following a priori estimate: for every $\varepsilon>0$ there is $\delta>0$
such that for all $x,y\in K$,%
\begin{equation}
\left\vert f(x)-f(y)\right\vert \leq\varepsilon+\delta\left\Vert
x-y\right\Vert ^{2}. \label{eq1}%
\end{equation}

Necessarily, such a function is uniformly continuous. The converse fails (for
example, see the case of the function $\sqrt{x}$ on $(0,\infty)),$ but it
occurs for all uniformly continuous and bounded functions $f:K\rightarrow
\mathbb{R}$. Examples of unbounded functions on $\mathbb{R}^{N}$ playing the
property of Korovkin absolute continuity are $\pm\operatorname*{pr}%
\nolimits_{1},...,\pm\operatorname*{pr}\nolimits_{N}$ and $\sum_{k=1}%
^{N}\operatorname*{pr}\nolimits_{k}^{2},$ where $\operatorname*{pr}%
\nolimits_{k}$ denotes the canonical projection on the $k$th coordinate. More
details are available in \cite{N2009}.

Based on the concept of Korovkin absolute continuity, the present authors have
extended Korovkin's theorem to the framework of sublinear and monotone
operators acting on function spaces defined on appropriate subsets $K$ of
$\mathbb{R}^{N}.$ See \cite{Gal-Nic-Med}, \cite{Gal-Nic-Aeq},
\cite{Gal-Nic-RACSAM} and \cite{Gal-Nic-subm}.

The aim of the present paper is to further extend these results to the case of
sequences of nonlinear and monotone operators converging pointwise to an
operator of the same nature, possibly different from the identity. In the
linear case, results of this kind have been obtained (in order) by Wang
\cite{Wang}, Guessab and Schmeisser \cite{GS} and Popa \cite{Popa2022}.
According to Theorem 1 in \cite{Popa2022}, if $L_{n}$ $(n\in\mathbb{N}$) and
$L$ are positive linear operators from $C\left(  [0,1]\right)  $ into itself
such that $L(1)L(x^{2})=\left(  L(x\right)  )^{2}$ and $L(1)(t)>0$ for all
$t\in\lbrack0,1],$ then $L_{n}(f)\rightarrow L(f)$ uniformly for all $f\in
C\left(  [0,1]\right)  $ if and only if
\[
L_{n}(1)\rightarrow L(1),~L_{n}(x)\rightarrow L(x)\text{ and }L_{n}%
(x^{2})\rightarrow L(x^{2})\text{\quad uniformly on }[0,1].
\]

This is extended by our Theorem 3 (Section 5) to a large class of nonlinear
operators defined on a space $C(K),$ with $K$ a compact subset of the
Euclidean space $\mathbb{R}^{N}.$ The possibility to replace $C(K)$ with other
function spaces (such the Lebesgue spaces $L^{P}(K)$ with $p\in\lbrack
1,\infty))$ makes the statement of Theorem 4 (Section 5). Due to the linearity
of \thinspace$L$, the condition $L(1)L(x^{2})=\left(  L(x\right)  )^{2}$ in
Popa's Theorem 1 is invariant under any translation $x\rightarrow x+\alpha$, a
fact that makes its result to work for all spaces $C\left(  [a,b]\right)  $
(not just for $C\left(  [0,1]\right)  ).$ In the nonlinear case, this remark
doesn't work, and in the case of compact subsets $K$ not included in
$\mathbb{R}_{+}^{N}$ we have to adjust the hypotheses by using an analogue of
$L(1)L(\left(  x+\alpha\right)  ^{2})=\left(  L(x+\alpha\right)  )^{2}.$

For the convenience of the reader, we summarized in Section 2 some very basic
facts on ordered Banach spaces, usually omitted by most of the textbooks on
Banach lattices, but essential for understanding our approach. So is Theorem
\ref{thmC(K)}, that motivates why in this paper we consider only operators
with values in a $C(K)$ space.

Section 3 is a thorough presentation of the class nonlinear operators that
makes the subject of this paper. These are sublinear and monotone operators
acting on ordered Banach spaces of functions which verify the property of
translatability relative to the multiples of unity. They was introduced in
\cite{Gal-Nic-RACSAM}, motivated by our interest in Choquet's theory of
integration, but there are many other examples outside that theory, mentioned
in this paper.

Section 4 is devoted to a nonlinear generalization of a result due to Popa
\cite{Popa2022}, which proves to be essential in proving our new Korovkin type
results in Section 5. Examples illustrating our main results are exhibited in
Section 6. The paper ends with a short list of open problems.

\section{Preliminaries on ordered Banach spaces}

Most authors (see Aliprantis and Tourky \cite{AT2007} and Schaefer and Wolff
\cite{SW1999}) define the \emph{ordered Banach spaces} as the real Banach
spaces $E$ endowed with an order relation $\leq$ such that the following three
conditions are verified:%
\begin{align*}
(OBS1)\text{ }x  &  \leq y\text{ implies }x+z\leq y+z\text{ for all }x,y,z\in
E;\text{ and}\\
(OBS2)\text{ }x  &  \leq y\text{ implies }\lambda x\leq\lambda y\text{ for
}x,y\in E\text{ and }\lambda\in\mathbb{R}_{+}=[0,\infty).\\
(OBS3)\text{ }0  &  \leq x\leq y\text{ in }E\text{ implies }\left\Vert
x\right\Vert \leq\left\Vert y\right\Vert .
\end{align*}
As usually, \thinspace$x\leq y$ will be also denoted also $y\geq x$ and $x<y$
(equivalently $y>x$) will mean that $y\geq x$ and $x\neq y.$

For convenience, we will consider in this paper \emph{only} ordered Banch
spaces $E$ whose positive cones are closed (in the norm topology),
\emph{proper} $(-E_{+}\cap E_{+}=\left\{  0\right\}  )$ and \emph{generating}
$(E=E_{+}-E_{+})$.

An ordered vector space $E$ such that every pair of elements $x,y$ admits a
supremum $\sup\{x,y\}$ and an infimum is called a \emph{vector lattice}. In
this case for each $x\in E$ we can define $x^{+}=\sup\left\{  x,0\right\}  $
(the positive part of $x$), $x^{-}=\sup\left\{  -x,0\right\}  $ (the negative
part of $x$) and $\left\vert x\right\vert =\sup\left\{  -x,x\right\}  $ (the
modulus of $x$). $\ $We have $x=x^{+}-x^{-}$ and $\left\vert x\right\vert
=x^{+}+x^{-}.$ A Banach lattice is any real Banach space $E$ which is at the
same time a vector lattice and verifies the condition%
\[
\left\vert x\right\vert \leq\left\vert y\right\vert \text{ implies }\left\Vert
x\right\Vert \leq\left\Vert y\right\Vert .
\]

Most classical Banach spaces are actually Banach lattices. So are the
Euclidean space $\mathbb{R}^{N}$ and the discrete spaces $c_{0},$ $c$ and
$\ell^{p}$ for $1\leq p\leq\infty$ (endowed with the coordinate-wise
ordering). The same is true for the function spaces
\begin{align*}
C(K)  &  =\left\{  f:K\rightarrow\mathbb{R}:f\text{ continuous on the compact
Hausdorff space}K\right\} \\
C_{b}(X)  &  =\left\{  f:X\rightarrow\mathbb{R}:\text{ }f\text{ continuous and
bounded on the metric space }X\right\}  ,\\
\mathcal{U}C_{b}(X)  &  =\left\{  f:X\rightarrow\mathbb{R}:\text{ }f\text{
uniformly continuous and bounded}\right.  \text{ on the}\\
&  \left.  \text{metric space }X\right\}  ,\\
C_{0}(X)  &  =\left\{  f:X\rightarrow\mathbb{R}:\text{ }f\text{ continuous and
null to infinity on the locally }\right. \\
&  \left.  \text{compact Hausdorff space }X\right\}  ,
\end{align*}
each one endowed with the sup-norm $\left\Vert f\right\Vert _{\infty}=\sup
_{x}\left\vert f(x)\right\vert $ and the pointwise ordering. Other Banach
lattices of an utmost interest are the Lebesgue spaces $L^{p}(\mu)$
($p\in\lbrack1,\infty]),$ endowed with norm%
\[
\left\Vert f\right\Vert _{p}=\left\{
\begin{array}
[c]{cl}%
\left(  \int_{X}\left\vert f(x)\right\vert ^{p}\mathrm{d}\mu(x)\right)  ^{1/p}
& \text{if }p\in\lbrack1,\infty)\\
\operatorname*{esssup}\limits_{x\in X}\left\vert f(x)\right\vert  & \text{if
}p=\infty
\end{array}
\right.
\]
and the pointwise ordering modulo null sets.

It is well known that all norms on the $N$-dimensional real vector space
$\mathbb{R}^{N}$ are equivalent. When endowed with the sup norm and the
coordinate wise ordering, $\mathbb{R}^{N}$ can be identified (algebraically,
isometrically and in order) with the Banach lattice $C\left(  \left\{
1,...,N\right\}  \right)  $, where $\left\{  1,...,N\right\}  $ carries the
discrete topology.

A convenient way to emphasize the properties of ordered Banach spaces is that
described by Davies in \cite{Davies1968}. Davies calls an ordered Banach space
$E$ \emph{regularly ordered }if%
\[
\left\Vert x\right\Vert =\inf\left\{  \left\Vert y\right\Vert :y\in E,\text{
}-y\leq x\leq y\right\}  \text{\quad for all }x\in E.
\]
This class of spaces brings together the Banach lattices and some other spaces
which are not vector lattices, a notorious example being $\operatorname*{Sym}%
(n,\mathbb{R)}$, the ordered Banach space of all $n\times n$-dimensional
symmetric matrices with real coefficients with the ordering%
\[
A\leq B\text{ if and only if }\langle Ax,x\rangle\leq\langle Bx,x\rangle
\]
and the norm%
\[
\left\Vert A\right\Vert =\sup_{\left\Vert x\right\Vert \leq1}\left\vert
\langle Ax,x\rangle\right\vert .
\]
For details, see \cite{NP2018}, Section 2.5, pp. 97-103.

\begin{lemma}
\label{lem1}Every ordered Banach space can be renormed by an equivalent norm
to become a regularly ordered Banach space.
\end{lemma}

For details, see Namioka \cite{Nam}. Some other useful properties of ordered
Banach spaces are listed below.

\begin{lemma}
\label{lem2}Suppose that $E$ is a regularly ordered Banach space. Then:

$(a)$ There exists a constant $C>0$ such that every element $x\in E$ admits a
decomposition of the form $x=u-v$ where $u,v\in E_{+}$ and $\left\Vert
u\right\Vert ,\left\Vert v\right\Vert \leq C\left\Vert x\right\Vert .$

$(b)$ The dual space of $E,$ $E^{\ast},$ when endowed with the dual cone
\[
E_{+}^{\ast}=\left\{  x^{\ast}\in E^{\ast}:x^{\ast}(x)\geq0\text{ for all
}x\in E_{+}\right\}
\]
is a regularly ordered Banach space.

$(c)\ x\leq y$ in $E$ is equivalent to $x^{\ast}(x)\leq x^{\ast}(y)$ for all
$x^{\ast}\in E_{+}^{\ast}.$

$(d)$ $\left\Vert x\right\Vert =\sup\left\{  x^{\ast}(x):x^{\ast}\in
E_{+}^{\ast},\text{ }\left\Vert x^{\ast}\right\Vert \leq1\right\}  $ for all
$x\in E_{+}.$
\end{lemma}

\begin{proof}
The assertion $(a)$ follows immediately from Lemma \ref{lem1}. For $(b)$, see
Davies \cite{Davies1968}, Lemma 2.4. The assertion $(c)$ is an easy
consequence of the Hahn-Banach separation theorem; see \cite{NP2018}, Theorem
2.5.3, p. 100.

The assertion $(d)$ is also a consequence of the Hahn-Banach separation
theorem; see \cite{SW1999}, Theorem 4.3, p. 223.
\end{proof}

The concept of strictly positive function admits a natural extension in the
framework of ordered Banach spaces. Precisely, a positive element of an
ordered Banach space $E$ is said to be \emph{strictly positive }provided that
\[
x^{\ast}(x)>0\text{ for every nonzero functional }x^{\ast}\in E_{+}^{\ast}.
\]

Strictly positive elements exists in every separable ordered Banach space $E$.
For example, choose a sequence $(x_{n})_{n}$ of elements of $E,$ dense in the
closed unit ball and consider their decomposition $x_{n}=u_{n}-v_{n}$
according to Lemma \ref{lem2} (a). Then $x=\sum2^{-n}(u_{n}+v_{n})$ is a
strictly positive element since any functional $x^{\ast}\in E_{+}^{\ast}$
vanishing at $x$ will vanish on all elements $x_{n}.$ By density, this implies
that $x^{\ast}=0.$ Actually a stronger results holds, precisely, the set of
all strictly positive elements of a separable ordered Banach space is dense
into the positive cone. See \cite{SW1999}, Theorem 7.6, p. 241.

The spaces $C(K)$, $C_{b}(X)$ and $\mathcal{U}C_{b}(X)$ (separable or not)
admit strictly positive elements with stronger properties, called order units.
Recall that an element $u>0$ of an ordered Banach space $E$ is called an
\emph{order unit} if for each for each $x\in E$ there exists a real number
$\lambda>0$ such that $-\lambda u\leq x\leq\lambda u.$ Necessarily, an order
unit is a strictly positive element. Besides the three aforementioned spaces,
the case of the space $\operatorname*{Sym}(n,\mathbb{R)}$ outlines the
importance of regularly ordered Banach spaces with an order unit whose norms
are associated to the unit via the formula%
\[
\left\Vert x\right\Vert _{u}=\inf\left\{  \lambda>0:-\lambda u\leq
x\leq\lambda u\right\}  ,\text{\quad}x\in E.
\]

The natural order unit of the vector lattice $\mathbb{R}^{N}$ is the vector
$u,$ whose all components equal 1. The norm associated to this unit is the sup-norm.

An example of Banach lattice without strictly positive elements is $C_{0}(X)$,
provided that $X$ is a nonseparable metric space.

The following result explains why in this paper we consider only operators
with values in $C(K)$ space.

\begin{theorem}
\label{thmC(K)}Every ordered Banach space $E$ can be represented as a vector
subspace of the space $C(K)$ of continuous real-valued functions on a compact
Hausdorff space $K$ via an order-preserving linear and continuous map
$\Phi:E\rightarrow C(K).$ If $E$ has a strictly positive element $e,$ then one
can choose $K$ such that $\Phi(e)=1,$ the unity of $C(K).$
\end{theorem}

\begin{proof}
According to the Alaoglu theorem, the set $K=\left\{  x^{\ast}\in E_{+}^{\ast
}:\left\Vert x^{\ast}\right\Vert \leq1\right\}  $ is compact relative to the
$w^{\ast}$ topology. Taking into account the assertions $(c)$ and $(d)$ of
Lemma \ref{lem2} one can easily conclude that $E$ embeds into $C(K)$ via the
positive linear isometry
\[
\Phi:E\rightarrow C(K),\text{\quad}\left(  \Phi(x)\right)  (x^{\ast})=x^{\ast
}(x).
\]
For the second part of Theorem \ref{thmC(K)}, notice that $\Phi(e)$ is a
strong order unit for $C(K).$ Then the conclusion follows from a classical
result due to Kadison \cite{Kad}, stating that every ordered real vector space
with an strong order unit can be represented as a vector subspace of the space
of continuous real-valued functions on a compact Hausdorff space via an
order-preserving map that carries the order unit to the constant function 1.
\end{proof}

\section{Weakly nonlinear operators acting on ordered Banach spaces}

Our next goal is to describe a class of nonlinear operators which provides a
convenient framework for the extension of Korovkin's theorem.

Given a metric space $X,$ we attach to it the vector lattice $\mathcal{F}(X)$
of all real-valued functions defined on $X$, endowed with the metric $d$ and
the pointwise ordering.

Suppose that $X$ and $Y$ are two metric spaces and $E$ and $F$ are
respectively ordered vector subspaces (or subcones of the positive cones) of
$\mathcal{F}(X)$ and $\mathcal{F}(Y)$ and that $\mathcal{F}(X)$ contains the
unity. An operator $T:E\rightarrow F$ is said to be a \emph{weakly nonlinear}
if it satisfies the following two conditions:

\begin{enumerate}
\item[(SL)] (\emph{Sublinearity}) $T$ is subadditive and positively
homogeneous, that is,%
\[
T(f+g)\leq T(f)+T(g)\quad\text{and}\quad T(\alpha f)=\alpha T(f)
\]
for all $f,g$ in $E$ and $\alpha\geq0;$

\item[(TR)] (\emph{Translatability}) $T(f+\alpha\cdot1)=T(f)+\alpha T(1)$ for
all functions $f\in E$ and all numbers $\alpha\geq0.$
\end{enumerate}

In the case when $T$ is \emph{unital} (that is, $T(1)=1)$ the condition of
translatability takes the form%
\[
T(f+\alpha\cdot1)=T(f)+\alpha1,
\]
for all $f\in E$ and $\alpha\geq0.$

A stronger condition than translatability is

\begin{enumerate}
\item[(TR$^{\ast}$)] (\emph{Strong translatability}) $T(f+\alpha
\cdot1)=T(f)+\alpha T(1)$ for all functions $f\in E$ and all numbers
$\alpha\in\mathbb{R}.$
\end{enumerate}

The last condition occurs naturally in the context of Choquet's integral,
being a consequence of what is called there the property of \emph{comonotonic
additivity}, that is,

\begin{enumerate}
\item[(CA)] $T(f+g)=T(f)+T(g)$ whenever the functions $f,g\in E$ are
comonotone in the sense that%
\[
(f(s)-f(t))\cdot(g(s)-g(t))\geq0\text{\quad for all }s,t\in X.
\]
See \cite{Gal-Nic-Aeq} and \cite{Gal-Nic-JMAA}, as well as the references therein.
\end{enumerate}

In this paper we are especially interested in those weakly nonlinear operators
which preserve the ordering, that is, which verify the following condition:

\begin{enumerate}
\item[(M)] (\emph{Monotonicity}) $f\leq g$ in $E$ implies $T(f)\leq T(g)$ for
all $f,g$ in $E.$
\end{enumerate}

\begin{remark}
\label{rem1} If $T$ is a weakly nonlinear and monotone operator, then
\[
T(\alpha\cdot1)=\alpha\cdot T(1)\text{\quad for all }\alpha\in\mathbb{R}.
\]
Indeed, for $\alpha\geq0$ the property follows from positive homogeneity.
Suppose now that $\alpha<0.$ Since $T(0)=0$ and $-\alpha>0$, by
translatability it follows that $0=T(0)=T(\alpha\cdot1+(-\alpha\cdot
1))=T(\alpha\cdot1)+(-\alpha)T(1)$, which implies $T(\alpha)=\alpha T(1)$.
\end{remark}

Examples weakly nonlinear and monotone operators can be found in
\cite{Gal-Nic-Aeq}, \cite{Gal-Nic-JMAA} and \cite{Gal-Nic-RACSAM}.

Suppose that $E$ and $F$ are two ordered Banach spaces and $T$ $:E\rightarrow
F$ is an operator (not necessarily linear or continuous).

If $T$ is positively homogeneous, then
\[
T(0)=0.
\]
As a consequence,
\[
-T(-f)\leq T(f)\text{\quad for all }f\in E
\]
and every positively homogeneous and monotone operator $T$ maps positive
elements into positive elements, that is,%
\begin{equation}
Tf\geq0\text{\quad for all }f\geq0. \label{pos-op}%
\end{equation}
Therefore, for linear operators the property (\ref{pos-op}) is equivalent to monotonicity.

Every sublinear operator is convex and a convex operator is sublinear if and
only if it is positively homogeneous.

The \emph{norm} of a continuous sublinear operator $T:E\rightarrow F$ can be
defined via the formulas%
\begin{align*}
\left\Vert T\right\Vert  &  =\inf\left\{  \lambda>0:\left\Vert T\left(
f\right)  \right\Vert \leq\lambda\left\Vert f\right\Vert \text{ for all }f\in
E\right\} \\
&  =\sup\left\{  \left\Vert T(f)\right\Vert :f\in E,\text{ }\left\Vert
f\right\Vert \leq1\right\}  .
\end{align*}
A sublinear operator may be discontinuous, but when it is continuous, it is
Lipschitz continuous. More precisely, if $T:E\rightarrow F$ is a continuous
sublinear operator, then
\[
\left\Vert T\left(  f\right)  -T(g)\right\Vert \leq2\left\Vert T\right\Vert
\left\Vert f-g\right\Vert \text{\quad for all }f\in E.
\]

Remarkably, all sublinear and monotone operators are Lipschitz continuous:

\begin{theorem}
\label{thmKrein}Every sublinear and monotone operator $T$ $:E\rightarrow F$
verifies the inequality
\[
\left\vert T(f)-T(g)\right\vert \leq T\left(  \left\vert f-g\right\vert
\right)  \text{\quad for all }f,g\in E
\]
and thus it is Lipschitz continuous with Lipschitz constant equals to
$\left\Vert T\right\Vert ,$ that is,
\[
\left\Vert T(f)-T(g)\right\Vert \leq\left\Vert T\right\Vert \left\Vert
f-g\right\Vert \text{\quad for all }f,g\in E.
\]

\end{theorem}

See \cite{Gal-Nic-subm} for details. Theorem \ref{thmKrein} is a
generalization of a classical result of M. G. Krein concerning the continuity
of positive linear functionals. See \cite{AA2001}.

\section{An a priori estimate}

The proof of our main results depend on the following a priori estimate,
previously noticed by Popa \cite{Popa2022} in the context of linear and
positive operators defined on the Banach lattice $C\left(  [a,b]\right)  .$

In what follows
\[
u=(1,...,1)
\]
denotes the natural order unit of the vector lattice $\mathbb{R}^{N}.$ Every
compact subset $K$ of the Euclidean space $\mathbb{R}^{N}$ can be moved into
the positive cone $\mathbb{R}_{+}^{N}$ via a translation of the form
$T_{a}:x\rightarrow x+\alpha u,$ associated to a number $\alpha\geq0.$ The
smallest $\alpha$ doing this job will be denoted $\alpha(K,u)$ and we will
refer to it as the \emph{deficit of positivity} of $K$ in the direction $u.$
We have%
\[
\alpha(K,u)=\inf\left\{  \alpha\geq0:K+\alpha u\subset\mathbb{R}_{+}%
^{N}\right\}
\]
and $\alpha(K,u)=0$ if $K\subset\mathbb{R}_{+}^{N}.$

\begin{lemma}
\label{lem3}Suppose that $K$ is a compact subset of the Euclidean space
$\mathbb{R}^{N},$ $X$ is a compact Hausdorff space and $V$ and $A$ are weakly
nonlinear and monotone operators from $C(K)$ into $C(X).$ Then for every
function $f\in C(K),$ and every $\varepsilon>0$ there exists $\delta>0$ such
that
\begin{multline*}
|T(f)A(1)-T(1)A(f)|\\
\leq\varepsilon T(1)A(1)+\delta\left\{  A(1)[T(\sum_{k=1}^{N}\left(
\operatorname*{pr}\nolimits_{k}+\alpha\right)  ^{2})-A(\sum_{k=1}^{N}\left(
\operatorname*{pr}\nolimits_{k}+\alpha\right)  ^{2})]\right. \\
+A(\sum_{k=1}^{N}\left(  \operatorname*{pr}\nolimits_{k}+\alpha\right)
^{2})[T(1)-A(1)]\\
+\left.  2[A(1)A(\sum_{k=1}^{N}\left(  \operatorname*{pr}\nolimits_{k}%
+\alpha\right)  ^{2})-\sum_{k=1}^{N}(A(-\operatorname*{pr}\nolimits_{k}%
-\alpha))^{2}]\right. \\
+\left.  2\sum_{k=1}^{N}A(-\operatorname*{pr}\nolimits_{k}-\alpha
)[A(-\operatorname*{pr}\nolimits_{k}-\alpha)-T(-\operatorname*{pr}%
\nolimits_{k}-\alpha)]\right\}  ,
\end{multline*}
whenever $\alpha\geq\alpha(K,u)$.
\end{lemma}

\begin{proof}
Let $f\in C(K)$ and $\varepsilon>0.$ Since $f$ is Korovkin absolutely
continuous, there is $\delta>0$ such that for all $x,y\in K$ we have
\begin{align*}
|f(x)-f(y)|  &  \leq\varepsilon+\delta\left\Vert x-y\right\Vert ^{2}\\
&  =\varepsilon+\delta\left\Vert x+\alpha u-\left(  y+\alpha u\right)
\right\Vert ^{2}.
\end{align*}
Viewing $y$ as a parameter, the last inequality can be rewritten as%
\[
\left\vert f-f(y)\right\vert \leq\varepsilon+\delta\left[  \sum_{k=1}%
^{N}\left(  \operatorname*{pr}\nolimits_{k}+\alpha\right)  ^{2}-2\sum
_{k=1}^{N}\left(  \operatorname*{pr}\nolimits_{k}(y)+\alpha\right)  \left(
\operatorname*{pr}\nolimits_{k}+\alpha\right)  \right.  \left.  +\sum
_{k=1}^{N}\left(  \operatorname*{pr}\nolimits_{k}(y)+\alpha\right)
^{2}\right]  .
\]

The choice of $\alpha$ makes $\operatorname*{pr}\nolimits_{k}+\alpha\geq0$ and
this allows us to value the properties of positive homogeneity and
translatability of the operator $T.$

Suppose for a moment that $f$ is nonnegative. Then, according to Theorem
\ref{thmKrein},
\begin{multline*}
|T(f)-f(y)T(1)|\leq T(|f-f(y)\cdot1|)\\
\leq\varepsilon T(1)+\delta\left[  T(\sum_{k=1}^{N}\left(  \operatorname*{pr}%
\nolimits_{k}+\alpha\right)  ^{2})+2\sum_{k=1}^{N}\left(  \operatorname*{pr}%
\nolimits_{k}(y)+\alpha\right)  T\left(  -\operatorname*{pr}\nolimits_{k}%
-\alpha\right)  \right. \\
\left.  +\sum_{k=1}^{N}\left(  \operatorname*{pr}\nolimits_{k}(y)+\alpha
\right)  ^{2}T(1)\right]  ,
\end{multline*}
which yields, for each $t\in X,$ the following inequality in $C(X):$
\begin{multline*}
|\left(  T(f)\right)  (t)-T(1)(t)\cdot f|\\
\leq\varepsilon\left(  T(1)\right)  (t)+\delta\left[  T(\sum_{k=1}^{N}\left(
\operatorname*{pr}\nolimits_{k}+\alpha\right)  ^{2})(t)+2\sum_{k=1}^{N}\left(
T\left(  -\operatorname*{pr}\nolimits_{k}-\alpha\right)  \right)  (t)\right.
\left(  \operatorname*{pr}\nolimits_{k}+\alpha\right) \\
\left.  +\left(  T(1)\right)  (t)\sum_{k=1}^{N}\left(  \operatorname*{pr}%
\nolimits_{k}+\alpha\right)  ^{2}\right]
\end{multline*}
Applying the weakly nonlinear and monotone operator $A$ to the both sides of
the last inequality and taking into account that $-2T(-\operatorname*{pr}%
\nolimits_{k}-\alpha)\geq0,$ we obtain
\begin{multline*}
|\left(  T(f)\right)  (t)A(1)-\left(  T(1)\right)  (t)A(f)|\leq A\left(
|\left(  T(f)\right)  (t)-\left(  T(1)\right)  (t)\cdot f|\right) \\
\leq\varepsilon\left(  T(1)\right)  (t)A(1)+\delta\left[  T(\sum_{k=1}%
^{N}\left(  \operatorname*{pr}\nolimits_{k}+\alpha\right)  ^{2})(t)A(1)\right.
\\
-2\sum_{k=1}^{N}T\left(  -\operatorname*{pr}\nolimits_{k}-\alpha\right)
(t)A(-\operatorname*{pr}\nolimits_{k}-\alpha)\left.  +\left(  T(1)\right)
(t)A(\sum_{k=1}^{N}\left(  \operatorname*{pr}\nolimits_{k}+\alpha\right)
^{2})\right]  ,
\end{multline*}
whence
\begin{multline*}
|\left(  T(f)\right)  (t)A(1)(t)-\left(  T(1)\right)  (t)A(f)(t)|\\
\leq\varepsilon\left(  T(1)\right)  (t)A(1)(t)+\delta\left\{  \lbrack
(T(\sum_{k=1}^{N}\left(  \operatorname*{pr}\nolimits_{k}+\alpha\right)
^{2})(t)-A(\sum_{k=1}^{N}\left(  \operatorname*{pr}\nolimits_{k}%
+\alpha\right)  ^{2})(t)]A(1)(t)\right. \\
+(A(\sum_{k=1}^{N}\left(  \operatorname*{pr}\nolimits_{k}+\alpha\right)
^{2})(t))\left[  T(1)(t)-A(1)(t)\right] \\
+2[A(1)(t)A(\sum_{k=1}^{N}\left(  \operatorname*{pr}\nolimits_{k}%
+\alpha\right)  ^{2})(t)-\sum_{k=1}^{N}\left(  A(-\operatorname*{pr}%
\nolimits_{k}-\alpha)\right)  ^{2}(t)]\\
\left.  +2\sum_{k=1}^{N}A(-\operatorname*{pr}\nolimits_{k}-\alpha)(t)\left[
A(-\operatorname*{pr}\nolimits_{k}-\alpha)(t)-T\left(  -\operatorname*{pr}%
\nolimits_{k}-\alpha\right)  (t)\right]  \right\}  .
\end{multline*}
As $t\in T$ was arbitrarily chosen, this ends the proof in the case of
nonnegative functions.

The case where $f$ is a signed function can be reduced to this one by
replacing $f$ by $f+\left\Vert f\right\Vert _{\infty}.$ Indeed, using the
property of weak nonlinearity of the operators $T$ and $A$ we have%
\begin{align*}
&  T(f+\left\Vert f\right\Vert _{\infty})A(1)-A(f+\left\Vert f\right\Vert
_{\infty})T(1)\\
&  =\left(  T(f)+\left\Vert f\right\Vert _{\infty}T(1)\right)  A(1)-\left(
A(f)+\left\Vert f\right\Vert _{\infty}A(1)\right)  T(1)\\
&  =T(f)A(1)-A(f)T(1).
\end{align*}

\end{proof}

\begin{corollary}
\label{cor1}Under the hypotheses of Lemma \emph{\ref{lem3}}, the following
inequality holds%
\begin{multline*}
\left\Vert T(f)A(1)-T(1)A(f)\right\Vert \\
\leq\varepsilon\Vert T(1)\Vert\cdot\Vert A(1)\Vert+\delta\left\{  \lbrack\Vert
A(1)\Vert\cdot\Vert T(\sum_{k=1}^{N}\left(  \operatorname*{pr}\nolimits_{k}%
+\alpha\right)  ^{2})-A(\sum_{k=1}^{N}\left(  \operatorname*{pr}%
\nolimits_{k}+\alpha\right)  ^{2})\Vert\right. \\
\left.  +\Vert A(\sum_{k=1}^{N}\left(  \operatorname*{pr}\nolimits_{k}%
+\alpha\right)  ^{2})\Vert\cdot\Vert T(1)-A(1)\Vert\right.  +\left.  2\Vert
A(1)A(\sum_{k=1}^{N}\left(  \operatorname*{pr}\nolimits_{k}+\alpha\right)
^{2})-\sum_{k=1}^{N}(A(-\operatorname*{pr}\nolimits_{k}-\alpha))^{2}%
\Vert\right. \\
+\left.  2\sum_{k=1}^{N}\Vert A(-\operatorname*{pr}\nolimits_{k}-\alpha
)\Vert\cdot\Vert T(-\operatorname*{pr}\nolimits_{k}-\alpha
)-A(-\operatorname*{pr}\nolimits_{k}-\alpha)\Vert\right\}  .
\end{multline*}

\end{corollary}

\section{The main results}

We are now in a position to settle the case of uniform convergence.

\begin{theorem}
\label{thm3}Suppose that $K$ is a compact subset of the Euclidean space
$\mathbb{R}^{N}$, $X$ is a compact Hausdorff space and $T_{n}$ $(n\in
\mathbb{N})$ and $A$ are weakly nonlinear and monotone operators from $E=C(K)$
into $C(X)$ such that
\begin{equation}
A(1)\text{ is a strictly positive element} \label{hyp>0}%
\end{equation}
and
\begin{equation}
A(1)A(\sum_{k=1}^{N}\left(  \operatorname*{pr}\nolimits_{k}+\alpha\right)
^{2})=\sum_{k=1}^{N}(A(-\operatorname*{pr}\nolimits_{k}-\alpha))^{2}
\label{hypA}%
\end{equation}
for a suitable $\alpha\geq\alpha(K,u)$.

Then%
\[
\lim_{n\rightarrow\infty}\left\Vert T_{n}(f)-A(f)\right\Vert =0\text{\quad for
all }f\in C(K)
\]
if and only if this property occurs for each of the functions
\begin{equation}
1,-(\operatorname*{pr}\nolimits_{1}+\alpha),...,-(\operatorname*{pr}%
\nolimits_{N}+\alpha)\text{ and }\sum\nolimits_{k=1}^{N}\left(
\operatorname*{pr}\nolimits_{k}+\alpha\right)  ^{2}. \label{testset}%
\end{equation}
Moreover, if $K\subset\mathbb{R}_{+}^{N},$ then Theorem \ref{thm3} works for
$\alpha=0.$
\end{theorem}

\begin{proof}
The necessity part is clear. For sufficiency, notice first that we may
restrict ourselves to the case where $K\subset\mathbb{R}_{+}^{N}$ by
performing, if necessary, a change of variable of the form $y=x+\alpha u.$
This allows us to continue the proof with $\alpha=0.$

According to Corollary \ref{cor1}, for all $f\in C\left(  K\right)  $ and
$\varepsilon>0$ there exists $\delta>0$ such that for all $n\in\mathbb{N}$ we
have%
\begin{multline*}
\left\Vert T_{n}(f)A(1)-T_{n}(1)A(f)\right\Vert \\
\leq\varepsilon\Vert T_{n}(1)\Vert\cdot\Vert A(1)\Vert+\delta\left\{
\lbrack\Vert A(1)\Vert\cdot\Vert T_{n}(\sum_{k=1}^{N}\operatorname*{pr}%
\nolimits_{k}^{2})-A(\sum_{k=1}^{N}\operatorname*{pr}\nolimits_{k}^{2}%
)\Vert\right. \\
\left.  +\Vert A(\sum_{k=1}^{N}\operatorname*{pr}\nolimits_{k}^{2})\Vert
\cdot\Vert T_{n}(1)-A(1)\Vert\right.  +2||A(1)A(\sum_{k=1}^{N}%
\operatorname*{pr}\nolimits_{k}^{2})-\sum_{k=1}^{N}[A(-\operatorname*{pr}%
\nolimits_{k})]^{2}||\\
+\left.  2\sum_{k=1}^{N}\Vert A(-\operatorname*{pr}\nolimits_{k})\Vert
\cdot\Vert T_{n}(-\operatorname*{pr}\nolimits_{k})-A(-\operatorname*{pr}%
\nolimits_{k})\Vert\right\}  ,
\end{multline*}
whence, by using hypothesis (\ref{hypA}), we infer that%
\begin{multline*}
\left\Vert T_{n}(f)A(1)-T_{n}(1)A(f)\right\Vert \\
\leq\varepsilon\Vert T_{n}(1)\Vert\cdot\Vert A(1)\Vert+\eta_{\varepsilon
}\left\{  [\Vert A(1)\Vert\cdot\Vert T_{n}(\sum_{k=1}^{N}\operatorname*{pr}%
\nolimits_{k}^{2})-A(\sum_{k=1}^{N}\operatorname*{pr}\nolimits_{k}^{2}%
)\Vert\right. \\
\left.  +\Vert A(\sum_{k=1}^{N}\operatorname*{pr}\nolimits_{k}^{2})\Vert
\cdot\Vert T_{n}(1)-A(1)\Vert+2\sum_{k=1}^{N}\Vert A(-\operatorname*{pr}%
\nolimits_{k})\Vert\cdot\Vert T_{n}(-\operatorname*{pr}\nolimits_{k}%
)-A(-\operatorname*{pr}\nolimits_{k})\Vert\right\}  .
\end{multline*}

Since $\lim_{n\rightarrow\infty}T_{n}(f)=A(f)$ for $f\in\left\{
1,-\operatorname*{pr}\nolimits_{1},...,-\operatorname*{pr}\nolimits_{N}%
,\sum_{k=1}^{N}\operatorname*{pr}\nolimits_{k}^{2}\right\}  ,$ the last
inequality implies that%
\[
\lim_{n\rightarrow\infty}\Vert T_{n}(f)A(1)-T_{n}(1)A(f)\Vert=0.
\]
Combining this fact with the inequality
\[
\Vert T_{n}(f)A(1)-A(1)A(f)\Vert\leq\Vert T_{n}(f)A(1)-T_{n}(1)A(f)\Vert
\]%
\[
+||T_{n}(1)A(f)-A(1)A(f)\Vert
\]%
\[
\leq\Vert T_{n}(f)A(1)-T_{n}(1)A(f)\Vert+\Vert A(f)\Vert\cdot\Vert
T_{n}(1)-A(1)\Vert,
\]
we infer that%
\[
\lim_{n\rightarrow\infty}\Vert\lbrack T_{n}(f)-A(f)]A(1)\Vert=0.
\]
By our hypotheses the function $A(1)$ is strictly positive, which yields
$\inf_{t\in X}A(1)(t)>0.$ Since%
\[
\lbrack\inf_{t\in X}A(1)(t)]\cdot\Vert T_{n}(f)-A(f)\Vert\leq\Vert\lbrack
T_{n}(f)-A(f)]A(1)\Vert,
\]
it follows that%
\[
\lim_{n\rightarrow\infty}\Vert T_{n}(f)-A(f)\Vert=0
\]
and the proof is done.
\end{proof}

\begin{remark}
\label{rem2}In the particular case where $E=F=C(K)$ and $A$ is an algebra
homomorphism preserving the unit $($in particular, if $A=I),$ the conditions
\emph{(\ref{hypA})} and \emph{(\ref{hyp>0})} are automatically fulfilled and
Theorem \ref{thm3} is covered by our previous results in
\emph{\cite{Gal-Nic-Med}}.
\end{remark}

The condition $\lim_{n\rightarrow\infty}\left\Vert T_{n}(1)-A(1)\right\Vert
=0$ in Theorem \ref{thm3} implies the equicontinuity of the operators $T_{n}$
since $\left\Vert T_{n}\right\Vert =T_{n}(1)$ for all $n.$

The result of Theorem \ref{thm3} can be put in a more generality as follows:

\begin{theorem}
\label{thm4}Suppose that $K$ is a compact subset of the Euclidean space
$\mathbb{R}^{N}$, $E$ is an ordered Banach space that includes $C(K)$ in such
a way that the canonical inclusion $i_{K}:C(K)\rightarrow E$ is a positive
linear operator with dense image and $X$ is a compact Hausdorff space. If
$T_{n}$ $(n\in\mathbb{N})$ and $A$ are weakly nonlinear and monotone operators
from $E$ into $C(X)$ such that%
\begin{gather}
A(1)\text{ is a strictly positive element,}\label{strictpos}\\
A(1)A(\sum_{k=1}^{N}\left(  \operatorname*{pr}\nolimits_{k}+\alpha\right)
^{2})=\sum_{k=1}^{N}(A(-\operatorname*{pr}\nolimits_{k}-\alpha))^{2}\text{ for
some }\alpha\geq d(K,u), \label{hypAA}%
\end{gather}
and%
\[
\sup_{n}\left\Vert T_{n}\right\Vert <\infty,
\]
\noindent then $\lim_{n\rightarrow\infty}\left\Vert T_{n}(f)-A(f)\right\Vert
=0$ for all $f\in E,$ if and only if this property of convergence occurs for
the following set of functions:%
\begin{equation}
1,-(\operatorname*{pr}\nolimits_{1}+\alpha),...,-(\operatorname*{pr}%
\nolimits_{N}+\alpha)~\text{and}~\sum\nolimits_{k=1}^{N}\left(
\operatorname*{pr}\nolimits_{k}+\alpha\right)  ^{2}. \label{hypKor}%
\end{equation}

\end{theorem}

In particular, Theorem \ref{thm4} holds for the Banach lattices $E=L^{p}%
(\mu),$ where $\mu$ is a positive Borel measure on $K.$ The fact that
$L^{p}(\mu)$ includes $C(K)$ as a dense subspace can be covered from many
sources, e.g., \cite{Bog}, Corollary $4.2.2$, p. $252$.

\begin{proof}
As in the case of Theorem \ref{thm3}, for the sufficiency part we may assume
that $\alpha=0$ and $K\subset\mathbb{R}_{+}^{N}.$

Let $f\in E$ and $\varepsilon>0.$ Since the vector lattice $C(K)$ is dense
into $E$ , there exists $g\in C(K)$ such that
\begin{equation}
\left\Vert f-g\right\Vert <\varepsilon. \label{dens}%
\end{equation}
Necessarily, $g$ is Korovkin absolutely continuous, so there exists a number
$\delta>0$ such that
\[
|g(y)-g(x)|\leq\varepsilon+\delta\cdot\left\Vert y-x\right\Vert ^{2}%
\]
for all $x\ $and $y$ in $\mathbb{R}^{N}$. $\ $

According to Corollary \ref{cor1} applied to $T_{n}\circ i_{K}$ and $A\circ
i_{K}$ for $\alpha=0,$ we have%
\begin{multline*}
\left\Vert T_{n}(g)A(1)-T_{n}(1)A(g)\right\Vert \\
\leq\varepsilon\left\Vert T_{n}(1)A(1)\right\Vert +\delta\left\{  \left\Vert
A(1)[T_{n}(\sum_{k=1}^{N}\operatorname*{pr}\nolimits_{k}^{2})-A(\sum_{k=1}%
^{N}\operatorname*{pr}\nolimits_{k}^{2})\right\Vert \right. \\
+\left\Vert A(\sum_{k=1}^{N}\operatorname*{pr}\nolimits_{k}^{2})[T_{n}%
(1)-A(1)]\right\Vert +2\left\Vert A(1)A\left(  \sum_{k=1}^{N}%
\operatorname*{pr}\nolimits_{k}^{2}\right)  -\sum_{k=1}^{N}%
(A(-\operatorname*{pr}\nolimits_{k}))^{2}\right\Vert \\
+\left.  2\sum_{k=1}^{N}\left\Vert A(-\operatorname*{pr}\nolimits_{k}%
)[A(-\operatorname*{pr}\nolimits_{k})-T_{n}(-\operatorname*{pr}\nolimits_{k}%
)]\right\Vert \right\}  ,
\end{multline*}
whence, by denoting $M=\sup\left\Vert T_{n}(1\right\Vert $ and taking into
account the hypothesis (\ref{hypAA}), we obtain the inequality
\begin{multline*}
\left\Vert T_{n}(g)A(1)-T_{n}(1)A(g)\right\Vert \\
\leq\varepsilon M\left\Vert A(1)\right\Vert +\delta\left\{  \left\Vert
A(1)[T_{n}(\sum_{k=1}^{N}\operatorname*{pr}\nolimits_{k}^{2})-A(\sum_{k=1}%
^{N}\operatorname*{pr}\nolimits_{k}^{2})\right\Vert \right. \\
+\left\Vert A(\sum_{k=1}^{N}\operatorname*{pr}\nolimits_{k}^{2})[T_{n}%
(1)-A(1)]\right\Vert \\
+\left.  2\sum_{k=1}^{N}\left\Vert A(-\operatorname*{pr}\nolimits_{k}%
)[A(-\operatorname*{pr}\nolimits_{k})-T_{n}(-\operatorname*{pr}\nolimits_{k}%
)]\right\Vert \right\}  .
\end{multline*}

Thus, under the presence of condition (\ref{hypKor}), we have
\[
\lim_{n\rightarrow\infty}\Vert T_{n}(g)A(1)-T_{n}(1)A(g)\Vert=0.
\]
Since
\begin{align*}
\Vert\lbrack T_{n}(g)-A(g)]A(1)\Vert &  \leq\Vert T_{n}(g)A(1)-T_{n}%
(1)A(g)+T_{n}(1)A(g)-A(1)A(g)\Vert\\
&  \leq\Vert T_{n}(g)A(1)-T_{n}(1)A(g)\Vert+\Vert A(g)\Vert\Vert
T_{n}(1)-A(1)\Vert,
\end{align*}
this yields
\[
\lim_{n\rightarrow\infty}\Vert\lbrack T_{n}(g)-A(g)]A(1)\Vert=0.
\]
Proceeding as in the proof of Theorem \ref{thm3}, we infer that
\[
\lim_{n\rightarrow\infty}\left\Vert T_{n}(g)-A(g)\right\Vert =0
\]
for all $g\in C\left(  K\right)  $. This property can be transferred to $f$
due to the inequalities
\[
\left\Vert T_{n}(f)-A(f)\right\Vert \leq\Vert T_{n}(f)-T_{n}(g)\Vert+\Vert
T_{n}(g)-A(g)\Vert+\Vert A(g)-A(f)\Vert
\]%
\begin{align*}
&  \leq\sup_{n}\left\Vert T_{n}\right\Vert \cdot\Vert f-g\Vert+\Vert
T_{n}(g)-A(g)\Vert+\Vert A\Vert\cdot\Vert f-g\Vert\\
&  \leq2(\sup_{n}\left\Vert T_{n}\right\Vert +\left\Vert A\right\Vert
)\varepsilon+\Vert T_{n}(g)-A(g)\Vert.
\end{align*}
Therefore
\[
\lim_{n\rightarrow\infty}\Vert T_{n}(f)-A(f)\Vert=0,
\]
for all $f\in E$ and the proof is done.
\end{proof}

\begin{remark}
\label{rem3}When the property of translatability is replaced in Theorem
\ref{thm3} and Theorem \ref{thm4} by the property of strong translatability,
the set of test functions can be simplified as follows:
\[
1,-\operatorname*{pr}\nolimits_{1},...,-\operatorname*{pr}\nolimits_{N}\text{
and }\sum\nolimits_{k=1}^{N}\left(  \operatorname*{pr}\nolimits_{k}%
^{2}+2\alpha\operatorname*{pr}\nolimits_{k}\right)  .
\]

\end{remark}

As was noticed by Korovkin \cite{Ko1953}, \cite{Ko1960}, his theorem mentioned
in the Introduction also works when the unit interval is replaced by the unit
circle
\[
S^{1}=\left\{  (\cos\theta,\sin\theta):\theta\in\lbrack0,2\pi)\right\}
\subset\mathbb{R}^{2}.
\]
Noticing that in this case the cosine function represents the restriction of
$\operatorname*{pr}_{1}$ to $S^{1}$ and the sine function represents the
restriction of $\operatorname*{pr}_{2}$ to $S^{1},$ one can easily deduce the
following result from Theorem \ref{thm3}:

\begin{theorem}
\label{thmtrig}Suppose that $T_{n}$ $(n\in\mathbb{N})$ and $A$ are weakly
nonlinear and monotone operators from $E=C_{2\pi}(\mathbb{R})$ into itself
such that $A(1)=1$ and
\begin{equation}
A(3+2\cos+2\sin)=(A(-1-\cos))^{2}+(A(-1-\sin))^{2} \label{trig+}%
\end{equation}

Then%
\[
T_{n}(f)\rightarrow A(f)\text{\quad uniformly for all }f\in C_{2\pi
}(\mathbb{R})
\]
if and only if this property occurs for each of the following test functions,%
\[
1,~-1-\cos,~-1-\sin\text{ and }3+2\cos+2\sin.
\]

\end{theorem}

Here $C_{2\pi}(\mathbb{R})$ denotes the Banach lattice of all continuous
functions $f:\mathbb{R\rightarrow R},$ periodic, of period $2\pi,$ endowed
with the sup-norm.

\begin{remark}
Notice that the hypothesis \emph{(\ref{trig+})} is automatically fulfilled
when $A$ is an algebra homomorphism preserving the unit $($in particular, if
$A=I).$ When the property of translatability is strengthen to strong
translatability, then the set of test functions in Theorem \ref{thmtrig} can
be replaced by%
\[
1,~-\cos,~-\sin\text{ and }2\cos+2\sin.
\]

\end{remark}

The analogues of Theorem \ref{thmtrig} for some other cases of interest such
as the $2$-dimensional torus $S^{1}\times S^{1}$ and the $2$-dimensional
sphere $S^{2}$ are left to the reader.

Theorem \ref{thmtrig} (and the remarks following it) also works in the context
of Ces\`{a}ro convergence, that is, of the convergence of the form%
\[
\frac{1}{n}\sum\nolimits_{k=1}^{n}T_{k}(f)\rightarrow A.
\]

Concrete examples illustrating the above results are indicated in the next section.

Last but not the least, one can extend our results (following the model of
Theorem 2 in \cite{Gal-Nic-RACSAM}) using instead of the classical families of
test functions on $\mathbb{R}^{N}$ the separating functions, which allow us to
replace the a priori estimates of the form (\ref{eq1}) by inequalities that
work in the general context of metric spaces. For details concerning these
functions see \cite{N2009}.

\section{Examples}

In what follows we will need the following family of polynomials
\[
p_{n,k}(x)={\binom{n}{k}}x^{k}(1-x)^{n-k},\text{\quad}0\leq k\leq n,
\]
related to Bernstein's proof of the Weierstrass approximation theorem. As is
well known they verify a number of combinatorial identities such as%
\begin{align}
\sum_{k=0}^{n}\binom{n}{k}x^{k}(1-x)^{n-k}  &  =1\label{comb1}\\
\sum_{k=0}^{n}k\binom{n}{k}x^{k}(1-x)^{n-k}  &  =nx\label{comb2}\\
\sum_{k=0}^{n}k^{2}\binom{n}{k}x^{k}(1-x)^{n-k}  &  =nx(1-x+nx) \label{comb3}%
\end{align}
for all $x\in\mathbb{R}.$ See \cite{CN2014}, Theorem 8.8.1, p. 256).

\begin{example}
\label{ex1}Given a continuous function $\varphi:[0,1]\rightarrow\lbrack0,1]$
one associates to it the sequence of nonlinear operators $T_{n}:C\left(
[0,1]\right)  \rightarrow C\left(  [0,1]\right)  $ defined by the formulas
\[
T_{n}(f)(x)=\sum_{k=0}^{n}p_{n,k}(\varphi(x))\sup_{[k/(n+1)\leq t\leq
,(k+1)/(n+1)]}f(t).
\]
Clearly, these operators are sublinear, strongly translatable, monotone and
unital $($the last property being a consequence of formula \emph{(\ref{comb1}%
)}$)$. We will show \emph{(}using Theorem \emph{\ref{thm3})} that the sequence
$(T_{n})_{n}$ is pointwise convergent to the linear positive and unital
operator
\[
A:C\left(  [0,1]\right)  \rightarrow C\left(  [0,1]\right)  ,\text{\quad
}A(f)=f\circ\varphi.
\]
The condition \emph{(\ref{hyp>0})} is trivial since $A$ is unital. The
fulfillment of condition \emph{(\ref{hypA})} is a consequence of the fact that
$A$ is an algebra homomorphism preserving the unit; see Remark
\emph{\ref{rem2}}. Using the identities \emph{(\ref{comb1})-(\ref{comb3}),}
one can show that $T_{n}(f)\rightarrow A(f)$ uniformly for each of the
functions $1$, $-x$ and $x^{2}$\emph{;} this set of test functions suffices
since $[0,1]\subset\mathbb{R}_{+}.$ Therefore the operators $T_{n}$ and $A$
fulfill the hypotheses of Theorem \emph{\ref{thm3} }and we can conclude that
$T_{n}(f)\rightarrow A(f),$ uniformly for all $f\in C([0,1]).$
\end{example}

\begin{example}
\label{ex2}As above, $\varphi:[0,1]\rightarrow\lbrack0,1]$ is a continuous
function. Attached to it is the sequence of Kantorovich operators $K_{n}%
:L^{1}([0,1])\rightarrow L^{1}\left(  [0,1]\right)  ,$ defined by the formulas%
\[
(K_{n}f)(x)=(n+1)\sum_{k=0}^{n}p_{n,k}(\varphi(x))\int_{k/(n+1)}^{\left(
k+1\right)  /(n+1)}f(t)\mathrm{d}t
\]
and also the operator
\[
A:L^{1}\left(  [0,1]\right)  \rightarrow L^{1}\left(  [0,1]\right)
,\text{\quad}A(f)=f\circ\varphi.
\]
Clearly, these operators are linear, positive and unital. Also,%
\[
\left\Vert K_{n}f\right\Vert _{1}\leq\int_{0}^{1}\left\vert f(t)\right\vert
dt\text{\quad for all }f\in L^{1}([0,1])\text{ and }n\in\mathbb{N},
\]
and the operator $A$ verifies the technical condition \emph{(\ref{hypAA})} for
$\alpha=0$, that is,%
\[
A(1)A(x^{2})=\sum_{k=1}^{N}(A(x))^{2}.
\]

The operators $T_{n}:L^{1}([0,1])\rightarrow L^{1}\left(  [0,1]\right)  $
given by
\[
T_{n}(f)=\sup\left\{  K_{n}(f),K_{n+1}(f)\right\}
\]
are sublinear, monotone and strongly translatable. Simple computations
\emph{(}based on the identities \emph{(\ref{comb1})-(\ref{comb3})), }show
that
\begin{align*}
K_{n}(1)  &  =1,\\
K_{n}(-x)  &  =\frac{-1}{2(n+1)}\sum_{k=0}^{n}(2k+1)p_{n,k}(\varphi
(x))=-\frac{n\varphi(x)}{n+1}-\frac{1}{2(n+1)}%
\end{align*}
and%
\[
K_{n}(x^{2})=\frac{n\left(  n-1\right)  \varphi^{2}(x)}{\left(  n+1\right)
^{2}}+\frac{2n\varphi(x)}{\left(  n+1\right)  ^{2}}+\frac{1}{3\left(
n+1\right)  ^{2}},
\]
which imply that $T_{n}(f)\rightarrow A(f)$ uniformly for each of the
functions $f\in\left\{  1,~-x,\text{ }x^{2}\right\}  $\emph{ (}and thus in
the\emph{ }$L^{1}$-norm\emph{). }According to Theorem \emph{\ref{thm4}}, this
property of convergence occurs for all functions $f\in L^{1}\left(
[0,1]\right)  .$
\end{example}

\begin{example}
\label{ex3}Let $\varphi$ and $A$ as in the preceding example and consider the
sequence of Choquet-Kantorovich operators $T_{n}:L^{1}\left(  [0,1]\right)
\rightarrow L^{1}\left(  [0,1]\right)  $ defined by the formulas%
\[
T_{n}(f)(x)=\sum_{k=0}^{n}p_{n,k}(\varphi(x))\cdot\frac{(C)\int_{k/(m+1)}%
^{(k+1)/(n+1)}f(t)\mathrm{d}\mu(t)}{\mu([k/(n+1),(k+1)/(n+1)])},
\]
where the letter $C$ in front of the integral means that we are dealing with a
Choquet's integral, in our case, the integral with respect to the capacity
$\mu$ representing a distortion of the Lebesgue measure $m$ via the formula
$\mu(A)=g(m(A)),$ where $g:[0,1]\rightarrow\lbrack0,1]$ is strictly
increasing, differentiable and concave function such that $g(0)=0$ and
$g(1)=1$. The theory developed by Gal and Trifa \emph{\cite{GT}} assures the
applicability of Theorem \emph{\ref{thm4}} to conclude that $T_{n}%
(f)\rightarrow A(f)$ for all functions $f\in L^{1}\left(  [0,1]\right)  .$
\end{example}

Examples involving functions of several variables can be exhibited using
tensor products of operators. For instance, in the case of Example \ref{ex1}
by the tensor product method we get
\begin{multline*}
T_{n}(f)(x,y)=\sum_{k=0}^{n}\sum_{j=0}^{n}p_{n,k}(\varphi(x))p_{n,j}%
(\varphi(y))\\
\cdot\sup\{f(t,s);t\in\lbrack k/(n+1),(k+1)/(n+1)],v\in\lbrack
j/(n+1),(j+1)/(n+1)]\},
\end{multline*}
which for any $f\in C([0,1]\times\lbrack0,1])$ and $A(f)(x,y)=f(\varphi
(x),\varphi(y))$ satisfy $T_{n}(f)\rightarrow A(f)$ as $n\rightarrow\infty$,
uniformly on $[0,1]\times\lbrack0,1]$.

\section{Open problems}

We end our paper by mentioning few open problems that might be of interest to
our readers. The first one concerns a technical hypothesis made in our
Theorems \ref{thm3}-\ref{thmtrig}:

\begin{problem}
How large is the set of solutions of the functional equation%
\[
A(1)A(\sum_{k=1}^{N}\left(  \operatorname*{pr}\nolimits_{k}+\alpha\right)
^{2})=\sum_{k=1}^{N}(A(-\operatorname*{pr}\nolimits_{k}-\alpha))^{2},
\]
among all weakly nonlinear and monotone operators $A:C(K)\rightarrow C(X)?$
\end{problem}

We already noticed that any algebra homomorphisms preserving the unit $($in
particular, $A=I)$ is a solution. Is the converse true? The usual books on
functional equations (including that by Acz\'{e}l and Dhombres \cite{AD}) seem
silent in this case.

Not entirely surprising, one can prove results similar to those in this paper
by working with other classes of limit operators. The following one was
presented by the second named author in a talk given at the \emph{Fourth
Romanian} \emph{Itinerant Seminar on Mathematical Analysis and its
Applications}, Bra\c{s}ov, May 19-20, 2022:

\begin{theorem}
\label{thmW}Suppose that $K$ is a connected and compact subset of the
Euclidean space $\mathbb{R}^{N},$ $A:C(K)\rightarrow\mathbb{R}$ is a
$($possibly nonlinear$)$ monotone functional such that $A(1)=1$ and
$(T_{n})_{n}$ is a sequence of sublinear and monotone operators from $C(K)$
into $C(K)$ such that%
\[
T_{n}(f)(x)\rightarrow A(f)\cdot1\text{\quad in the sup-norm}%
\]
for each of the test functions $1,~\pm\operatorname*{pr}_{1},...,~\pm
\operatorname*{pr}_{N}~$and $\sum_{k=1}^{N}\operatorname*{pr}_{k}^{2}$. Then
this property of convergence extends to all nonnegative functions $f$ in
$C(K)$.

It occurs for all functions in $C(K)$ provided that the operators $T_{n}$ are
weakly nonlinear and monotone.

Moreover, in either case the family of test functions can be reduced to
$1,~-\operatorname*{pr}_{1},...,~\allowbreak-\operatorname*{pr}_{N}$ and
$\sum_{k=1}^{N}\operatorname*{pr}_{k}^{2}$ provided that $K$ is included in
the positive cone of $\mathbb{R}^{N}$.
\end{theorem}

Theorem \ref{thmW} is a nonlinear extension of the Weyl ergodic theorem. This
theorem, whose essence is the \emph{unique ergodicity }of the irrational
rotation%
\[
R_{\alpha}(x)=\left(  x+\alpha\right)  \text{ }\operatorname*{mod}2\pi
\quad(\alpha/\pi\notin\mathbb{R}\backslash\mathbb{Q}),
\]
asserts the uniform convergence of the ergodic averages $\frac{1}{n}%
\sum\nolimits_{k=0}^{n-1}f\circ(R_{\alpha})^{k}$ of any $f\in C_{2\pi
}(\mathbb{R})$ to the constant function $\frac{1}{2\pi}\int_{-\pi}^{\pi
}f\mathrm{d}\mu.$ See Oxtoby \cite{Oxt} and Parry \cite{Parry} for details.
The Weyl ergodic theorem represents the particular case of Theorem \ref{thmW}
when $C(K)=C_{2\pi}(\mathbb{R}),$%
\[
T_{n}(f)=\frac{1}{n}\sum\nolimits_{k=0}^{n-1}f\circ(R_{\alpha})^{k}\text{ and
}A(f)=\left(  \frac{1}{2\pi}\int_{-\pi}^{\pi}f\mathrm{d}\mu\right)  \cdot1.
\]
.

The classical proof of the Weyl ergodic theorem consists in verifying it in
the case of exponentials $e^{inx}$ and next in applying the trigonometric form
of the Weierstrass approximation theorem. Theorem (\ref{thmW}) reduces the set
of test functions to the triplet $1$, $\cos$ and $\sin!$

Notice that Theorem \ref{thmtrig} does not apply in this case because the
technical hypothesis (\ref{trig+}) fails.

\begin{problem}
Find other substitutes for the technical condition \emph{\ref{hypA}} that
avoids the presence of the operation of multiplication in the codomain.
\end{problem}

Another problem left open is the following one:

\begin{problem}
Find an extension of Theorem \emph{\ref{thm4}} to the case of operators
defined on the Lebesgue spaces $L^{p}(\mu)$ associated to a finite measure
$\mu$ on $\mathbb{R}^{N}.$ Of a special interest is the case of the Gaussian
measure $\frac{1}{\left(  \sqrt{2\pi}\right)  ^{N}}e^{-\left\Vert x\right\Vert
^{2}/2}\mathrm{d}x$.
\end{problem}


\begin{thebibliography}{99}                                                                                               %


\bibitem {AA2001}Abramovich, Y.A., Aliprantis, C.D.: Positive Operators.
Handbook of the Geometry of Banach spaces, I, pp. 85--122. North-Holland,
Amsterdam (2001)

\bibitem {AD}Acz\'{e}l, J., Dhombres, J.: Functional equations in several
variables with applications to mathematics, information theory and to the
natural and social sciences. Cambridge Univ. Press, Cambridge-New York-New
Rochelle-Melbourne-Sydney (1989).

\bibitem {AT2007}Aliprantis, C.D., Tourky R.: Cones and Duality. Graduate
Studies in Mathematics vol. \textbf{84}, American Mathematical Society,
Providence, R.I. (2007).

\bibitem {Alt2010}Altomare, F.: Korovkin-type theorems and positive operators.
Surveys in Approximation Theory. \textbf{6}, 92-164 (2010)

\bibitem {Alt2021}Altomare, F.: On positive linear functionals and operators
associated with generalized means. J. Math. Anal. Appl. \textbf{502}, paper
no. 125278 (2021)

\bibitem {Alt2021b}Altomare, F.: On the convergence of sequences of positive
linear operators and functionals on bounded function spaces. Proc. Amer. Math.
Soc. \textbf{149}, 3837--3848 (2021)

\bibitem {Alt2022}Altomare F.: Korovkin-type theorems and local approximation
problems. Expo. Math. (2022), doi: https://doi.org/10.1016/j.exmath.2022.06.001.

\bibitem {AC1994}Altomare, F., Campiti, M.: Korovkin-Type Approximation Theory
and Its Applications. de Gruyter Studies in Mathematics vol. \textbf{17},
Berlin (1994, reprinted 2011).

\bibitem {Bog}Bogachev, V.: Measure Theory Vol. I. Springer Science \&
Business Media, (2007)

\bibitem {CN2014}Choudary A.D.R., Niculescu, C.P.: Real Analysis on Intervals.
Springer, New Delhi (2014)

\bibitem {Davies1968}E. B. Davies, The structure and ideal theory of the
pre-dual of a Banach lattice, Trans. Amer. Math. Soc. \textbf{131}, 544--555 (1968)

\bibitem {Gal-Nic-Med}Gal, S.G., Niculescu, C.P.: A nonlinear extension of
Korovkin's theorem. Mediterr. J. Math. \textbf{17}, Article no. 145 (2020).

\bibitem {Gal-Nic-JMAA}Gal, S.G., Niculescu, C.P.: Choquet operators
associated to vector capacities. J. Math. Anal. Appl. \textbf{500}, article
no. 125153 (2021)

\bibitem {Gal-Nic-Aeq}Gal, S.G., Niculescu, C.P.: A note on the Choquet type
operators. Aequationes Math. \textbf{95}, 433--447 (2021)

\bibitem {Gal-Nic-RACSAM}Gal, S.G., Niculescu, C.P.: Nonlinear versions of
Korovkin's abstract theorems. Revista de la Real Academia de Ciencias Exactas,
F\~{A}-sicas y Naturales. Serie A. Matematicas. \textbf{116}, Article number
68 (2022)

\bibitem {Gal-Nic-subm}Gal, S.G., Niculescu, C.P.: Korovkin type theorems for
weakly nonlinear and monotone operators. Preprint available on arXiv:2206.14102

\bibitem {GT}Gal, S.G. Trifa, S.: Quantitative estimates for $L^{p}%
$-approximation by Bernstein-Kantorovich-Choquet polynomials with respect to
distorted Lebesgue measures. Constructive Math. Analysis. \textbf{2}, no. 1,
15--21 (2019)

\bibitem {GS}Guessab, A., Schmeisser, G.: Two Korovkin-type theorems in
multivariate approximation. Banach Journal of Mathematical Analysis
\textbf{2}(2), 121--128 (2008)

\bibitem {Kad}Kadison, R.V.: A representation theory for commutative
topological algebra. Mem. Amer. Math. Soc. no. \textbf{7}, 1951.

\bibitem {Ko1953}Korovkin, P.P.: On convergence of linear positive operators
in the space of continuous functions (Russian). Doklady Akad. Nauk. SSSR. (NS)
\textbf{90}, 961--964 (1953).

\bibitem {Ko1960}Korovkin, P.P.: Linear Operators and Approximation Theory.
Fitzmatgiz, Moscow (1959) (in Russian) [English translation, Hindustan Publ.
Corp., Delhi (1960)]

\bibitem {Nam}Namioka I.: Partially ordered linear topological spaces. Mem.
Amer. Math. Soc. No. \textbf{24} (1957).

\bibitem {N2009}Niculescu, C.P.: An overview of absolute continuity and its
applications. Internat. Ser. Numer. Math. \textbf{157}, 201--214.
Birkh\"{a}user, Basel (2009)

\bibitem {NP2018}Niculescu, C.P., Persson, L.-E.: Convex Functions and their
Applications. A Contemporary Approach,\ Second Edition, CMS Books in
Mathematics vol. \textbf{23}, Springer-Verlag, New York (2018)

\bibitem {Oxt}J.C. Oxtoby: Ergodic Sets. Bull. Amer. Math. Soc. \textbf{58},
116--136 (1952)

\bibitem {Parry}Parry, W.: Topics in ergodic theory. Cambridge tracts in
mathematics Vol. \textbf{75}. Cambridge University Press (1981)

\bibitem {Popa2022}Popa, D.: An operator version of Korovkin's theorem. J.
Math. Anal. Appl. \textbf{515}(1), Paper no. 126375 (2022)

\bibitem {SW1999}Schaefer, H.H., Wolff, M.P.: Topological Vector Spaces,
Graduate Texts in Mathematics \textbf{3}, Second Edition, Springer
Science+Business Media New York, (1999)

\bibitem {Wang}Wang, H., Korovkin-type theorem and application. J. Approx.
Theory \textbf{132}, 258--264 (2005)
\end{thebibliography}
\end{document}